\documentclass[11pt]{article}
\usepackage{amsfonts}
\usepackage{amsthm}
\usepackage{amsmath}
\usepackage{amssymb}
\usepackage{epsfig}
\usepackage[percent]{overpic}
\usepackage{graphics}
\usepackage{psfrag}

\usepackage{verbatim}
\usepackage{color}

\usepackage[all]{xy}
\usepackage{makeidx}
\usepackage{enumerate}
\usepackage[letterpaper,margin=0.9in]{geometry}

\usepackage{tikz}
\usetikzlibrary{matrix}

\let\oldbibliography\thebibliography
\renewcommand{\thebibliography}[1]{%
	\oldbibliography{#1}%
	\setlength{\itemsep}{-1.2mm}%
}

\theoremstyle{plain}
\newtheorem{thm}{Theorem}[section]

\newtheorem{lem}[thm]{Lemma}
\newtheorem{prop}[thm]{Proposition}

\theoremstyle{definition}
\newtheorem{defn}[thm]{Definition}

\newtheorem{ex}[thm]{Example}

\newtheorem{theorem}{Theorem}

\newtheoremstyle{myremark}
{3pt}
{3pt}
{\small \rmfamily}
{5pt}
{\rmfamily}
{:}
{.5em}
{}

\theoremstyle{myremark}

\def\R{\mathbb{R}}

\def\N{\mathbb{N}}
\def\Z{\mathbb{Z}}

\def\cB{\mathcal{B}}

\def\cM{\mathcal{M}}
\def\cN{\mathcal{N}}

\def\cT{\mathcal{T}}

\def\txtb{{\textnormal{b}}}

\def\txtd{{\textnormal{d}}}
\def\txte{{\textnormal{e}}}

\def\I{\infty}

\newcommand{\be}{\begin{equation}}
\newcommand{\ee}{\end{equation}}
\newcommand{\benn}{\begin{equation*}}
\newcommand{\eenn}{\end{equation*}}
\newcommand{\bea}{\begin{eqnarray}}
\newcommand{\eea}{\end{eqnarray}}
\newcommand{\beann}{\begin{eqnarray*}}
	\newcommand{\eeann}{\end{eqnarray*}}

\newcommand{\myendex}{$\blacklozenge$\end{ex}}
\newcommand{\myendexerc}{$\lozenge$\end{exerc}}
\newcommand{\myendpexerc}{$\lozenge$\end{pexerc}}

\newcommand{\abs}{{\sf abs}}
\newcommand{\Lip}{{\sf Lip}}
\newcommand{\BL}{{\sf BL}}
\newcommand{\TV}{{\sf TV}}
\newcommand{\rd}{\mathrm{d}}

\begin{document}
\numberwithin{equation}{section}
\author{Marios Antonios Gkogkas\thanks{
Department of Mathematics, Technical University
of Munich, 85748 Garching b.~M\"unchen, Germany}~,~Christian Kuehn\footnotemark[1]~~and~Chuang Xu\footnotemark[1]}

\title{Continuum Limits for Adaptive Network Dynamics}

\maketitle

\begin{abstract}
Adaptive (or co-evolutionary) network dynamics, i.e., when changes of the network/graph topology are coupled with changes in the node/vertex dynamics, can give rise to rich and complex dynamical behavior. Even though adaptivity can improve the modelling of collective phenomena, it often complicates the analysis of the corresponding mathematical models significantly. For non-adaptive systems, a possible way to tackle this problem is by passing to so-called continuum or mean-field limits, which describe the system in the limit of infinitely many nodes. Although fully adaptive network dynamic models have been used a lot in recent years in applications, we are still lacking a detailed mathematical theory for large-scale adaptive network limits. For example, continuum limits for static or temporal networks are already established in the literature for certain models, yet the continuum limit of fully adaptive networks has been open so far. In this paper we introduce and rigorously justify continuum limits for sequences of adaptive Kuramoto-type network models. The resulting integro-differential equations allow us to incorporate a large class of co-evolving graphs with high density. Furthermore, we use a very general measure-theoretical framework in our proof for representing the (infinite) graph limits, thereby also providing a structural basis to tackle even larger classes of graph limits. As an application of our theory, we consider the continuum limit of an adaptive Kuramoto model directly motivated from neuroscience and studied by Berner et al.~in recent years using numerical techniques and formal stability analysis.
\end{abstract}

\textbf{Keywords:} adaptive networks, co-evolutionary networks, continuum limit, graph limits, Kuramoto-type models.

\section{Introduction}

Adaptive (or co-evolutionary) networks are models of diverse phenomena, ranging from neuroscience, epidemiology, ecology to game theory, just to name a few~\cite{GB2008,GS2009}. The main feature of adaptive networks is the coupling of dynamics \emph{of} and \emph{on} the network. From the perspective of recent applications~\cite{Berner2019,BVSY21,GMT2016,HLLJ2018,HN2016}, it is understood that the feedback mechanism between the network topology and the node dynamics can lead to a wide variety of interesting effects. For example, bifurcation-induced transitions may change their direction from super-critical (soft) transitions to sub-critical (hard) transitions~\cite{GrossDLimaBlasius,KuehnBick}, which are referred to as tipping points or critical transitions. Also, new dynamical states such as multi-clusters can be induced by taking into adaptivity~\cite{Berner2019}. Although a number of formal asymptotic methods or moment-closure schemes exist to study adaptive networks mathematically~\cite{Burger21,KissMillerSimon,KuehnMC}, a rigorous mathematical development of the field has just been started. One natural way to study networks rigorously is to exploit the large-scale limit of an infinite number of nodes/vertices to derive continuum or mean-field differential equations. This approach is not yet available for adaptive networks. In this paper, we derive continuum limit differential equations for adaptive/co-evolutionary dynamics given by \emph{Kuramoto-type models}. The original Kuramoto model was posed on a static graph and for all-to-all coupling originally by Kuramoto~\cite{Kuramoto84,Kuramoto75}, who first used it to model synchronization of systems of coupled oscillators. We refer to~\cite{Strogatz2000} for a concise introduction to the original Kuramoto model and to~\cite{Acebronetal,Arenasetal,PikovskyRosenblumKurths} for a more detailed view on the paradigmatic status of Kuramoto models as well as for further references. A quite general class of adaptive Kuramoto-type models on $N$ oscillators/nodes can be written in the following form
\begin{subequations}
\label{coevolution}
\begin{alignat}{3}
\frac{\txtd \phi_i^N}{\txtd t}=\dot{\phi}^N_i=&~\omega_i^N(\phi_i^N,t)+\frac{1}{N}\sum_{j=1}^NW^N_{ij}(t)D(\phi_i^N,\phi_j^N),\quad t>t_0, \label{coevolution-a}\\
\frac{\txtd W_{ij}^N}{\txtd t}=\dot{W}_{ij}^N=&~-\varepsilon \left (W_{ij}^N+H(\phi_i^N,\phi_j^N) \right),\quad t>t_0, \label{coevolution-b} \\
\phi_i^N(t_0) =& ~\phi_i^{N,0}, \quad W^N_{ij}(t_0) = W^{N}_{ij}, \quad i,j \in \{1,\ldots,N\} =:[N],
\end{alignat} 	
\end{subequations}
where $t\in[t_0,t_0+T]=:\mathcal{T}_{t_0,T}$ is the time interval with $0<T$, $\phi_i^N=\phi_i^N(t) \in \mathbb{T}=\mathbb{R}/(2\pi\mathbb{Z})$ represents the phase of the $i$-th oscillator for $i \in [N]$, $\omega_i^N: \mathbb{T}\times\mathbb{R} \to \mathbb{R} $ is the vector field describing the intrinsic frequency, $D:\mathbb{T}^2 \to \mathbb{R}$ is a coupling function, $W^N(t) = (W^N_{i,j}(t))_{i,j =1,...N } \in \mathbb{R}^{N\times N}$ is the time-evolving adjacency matrix of the network of oscillators, which takes into account the local information of two interacting oscillators via the function $H: \mathbb{T}^2 \to \mathbb{R}$, and $\varepsilon>0$ is a parameter; the structure of the model~\eqref{coevolution} is directly motivated by recent applications~\cite{Berner2019,BVSY21}. The parameter $\varepsilon$ controls the time scale between the dynamics \emph{on} the network of the phases and the dynamics \emph{of} the network weights. In applications, one frequently takes $\varepsilon>0$ small so that the network topology evolves slowly in comparison to the node dynamics; we do not require this time scale separation for our results but we believe it to be insightful to track $\varepsilon$ throughout the calculations. Let us note that the weights $W^N_{i,j}(t)$ are not assumed to be symmetric, i.e., we do not assume that $W^N_{i,j}(t) = W^N_{j,i}(t)$ for all $i,j \in [N]$, nor do we assume positivity of the weights. This entails that our graphs can be directed, weighted, and signed.

We want to investigate the limit of~\eqref{coevolution} as the number of oscillators tends to infinity ($N\to\infty$). If a suitable limiting differential equations equation exists, it is often analytically and/or numerically easier to handle than systems with finite large $N$. This approach can make it feasible to analytically understand phase transitions and synchronization~\cite{Med2013,OWMS12}, chimera and clustered states~\cite{Kuramoto02}, or multistability~\cite{Wiley06}. There are different approaches to obtain a limiting equation. If the oscillator phases can be arranged on a limiting space and the pointwise limit exists as $N\to \infty$, then this is usually referred to as the \emph{continuum limit}. The other common approach is to aim for a \emph{mean-field limit}, which characterizes the weak limit of empirical distributions of the phases of oscillators. For the classical all-to-all coupling topology, many results on limiting equations exist, particularly for the mean-field limit~\cite{Lancellotti,Neunzert}. Thanks to new techniques for graph limits~\cite{Scededgy2018,L12}, there have been several works on mean-field or continuum limits also for systems without all-to-all coupling~\cite{GK20,Kuehn2020,KuehnThrom2,KuehnXu,Med2013,Med2013b,Med14}. In this work, we are interested in continuum limits for adaptive Kuramoto-type models. For the static network case, results with random weights were developed in~\cite{Med2013b}, while the analysis of the steady states on small-world graphs was carried out via continuum limits in~\cite{Med14}. In addition to the static restriction, another relevant assumption in~\cite{Med2013b,Med14} was the existence of a limiting graphon, which places a density requirement on the coupling structure. Furthermore, if there is no feedback between network topology and node dynamics but just a given time-dependent family of networks, continuum limits were considered in~\cite{APD2020}. A view via moment hierarchies and formal moment closure can be found in~\cite{Burger21}. In summary, the myriad application results for finite $N$ for adaptive Kuramoto-type models and the existing results for static/temporal network continuum limits for dense graphs, naturally pose the question, whether one can prove continuum limits for fully adaptive Kuramoto-type models on various classes of graph limits? In this paper, we answer this question positively and establish rigorous results leading to a continuum limit non-local integro-differential equation. We prove the well-posedness of the limit equation and pointwise approximation by solutions via a sequence of finite $N$ models of the form~\eqref{coevolution}. 

\subsection{The continuum limit}
\label{sec: intro cont lim}

To start, we need a limiting structure for the graph and fix $(X, \mathfrak{B}(X), \mu_X)$ as a Borel probability space. Here $X$ can be regarded as the vertex space of the graph limit with time-dependent edge weights. When one first tries to introduce a continuum limit for~\eqref{coevolution}, one realizes that it is unclear, how to even translate the evolution for the weights~\eqref{coevolution-b} into the language of graph limits. Ideally, we would like to be able to model a broad class of limiting graph topologies. One possible way is to use the mathematical space of graphops, which is a very elegant and recent mathematical framework for representing finite or infinite undirected graphs of arbitrary density and their graph limits~\cite{Scededgy2018}. Recall from~\cite{Scededgy2018} that an undirected graph with positive weights can be represented by a bounded, self-adjoint and positivity-preserving operator $A: L^\infty(X;\mu_X) \to L^1(X; \mu_X)$, where $A$ is called a \emph{graphop}. A graphop can be viewed as a generalized concept for the adjacency matrix as for finite-dimensional graphs it coincides with the adjacency matrix. Graphops are in direct correspondence to a family of finite measures $(\eta^x)_{x\in X}$, called \emph{fiber measures}, via the Riesz representation theorem
\begin{equation*}
(Af)(x) = \int_{X} f(y) ~\txtd\eta^x(y) \text{ \quad $\mu_X$-a.e.~$x \in X$,\quad  for $f \in L^\infty(X;\mu_X)$.}
\end{equation*}
Intuitively, for a node $x \in X$, the fiber measure $\eta^x$ defines the neighbors of $x$. Now, having in mind this representation of graphs via operators and corresponding fiber measures, we go a step further and view a general weighted graph on $X$ just as a family of finite measures $(\eta^x)_{x\in X}$. Then we embed the evolution for the weights \eqref{coevolution-b} in an evolution equation for time-dependent measures $(\eta^x_t)_{x\in X}$; note that in the literature on graph limits, one would use a sub-index $x$ to denote fiber measures of graphops but it seems more natural for our context to use a super-index for $x$ to have the sub-index more directly available for the time dependence.

Having motivated the measure-theoretical framework for representing graphs, we can now introduce the continuum limit of the adaptive Kuramoto-type model \eqref{coevolution}. Consider the initial value problem (IVP) of following family of integro-differential equations
\begin{subequations}
\label{characteristic-Eq-1}
\begin{alignat}{2}
\frac{\partial\phi(t,x)}{\partial t}=&~\omega(x, \phi(t,x),t)+\int_X D(\phi(t,x),\phi(t,y))~\rd \eta^x_{t}(y),\quad t>t_0, \label{characteristic-Eq-1a}\\
\frac{\partial \eta^x_t(y)}{\partial t}=&~-\varepsilon \eta^x_{t}(y)
-\varepsilon H(\phi(t,x),\phi(t,y)) \mu_X(y),\quad t>t_0, \label{characteristic-Eq-1b}\\
\phi(t_0,x)=&~\phi_0(x),\quad \eta^x_{t_0}=\eta^x_0,
\end{alignat}	
\end{subequations}
where $\omega: X \times \mathbb{T} \times \mathbb{R} \to \mathbb{R}$ is the frequency function and we consider initial conditions $(\phi_0,\eta_0)\in C(X,\mathbb{T}\times\cM(X))$ with $\cM(X)$ denoting the space of finite signed Borel measures. Here, \eqref{characteristic-Eq-1} is to be interpreted in the integral form:
\begin{subequations}
\label{characteristic-Eq-int-1}
\begin{alignat}{2}
\phi(t,x)=&~\phi_0(x)+\int_{t_0}^t\left[\omega(x, \phi(\tau,x), \tau)+\int_X D(\phi(\tau,x),\phi(\tau,y))~\rd \eta_{\tau}^x(y)
\right]~\rd\tau,\quad t>t_0,\\
\label{characteristic-Eq-int-2}
\eta^x_t(y)=&~\eta_0^x-\varepsilon\int_{t_0}^t \eta^x_{\tau}(y)~\rd\tau
-\varepsilon\int_{t_0}^t H\left(\phi(\tau,x),\phi(\tau,y)\right)\mu_X(y)~\rd\tau,\quad t>t_0,\\
\phi(t_0,x)=&~\phi_0(x),\quad \eta_{t_0}^x=\eta_0^x.
\end{alignat}
\end{subequations}
Moreover, note that \eqref{characteristic-Eq-1b} (\eqref{characteristic-Eq-int-2}, respectively) is an evolution equation for the time-evolving family of measures $(\eta^x_t)_{x\in X}$, which should be understood in the weak sense, meaning that for any bounded continuous test-function $f\in C_{\textnormal{b}}(X)$ we have
\begin{eqnarray}
\int_X f(y)~\rd\eta^x_t(y)&=&\int_X f(y)~\rd \eta_0^x(y) - \varepsilon\int_{t_0}^t \left(\int_X f(y)~\rd \eta_{\tau}^x(y)\right)~\rd\tau \nonumber \\
&&-\varepsilon\int_{t_0}^t\left(\int_Xf(y)H(\phi(\tau,y)-\phi(\tau,x))~\rd\mu_X(y)\right)~\rd\tau. \label{eq: weak form of eq. for meas}
\end{eqnarray}

\begin{defn}\textbf{(Solution for the IVP of the continuum limit)} \label{def: sol of the IVP} A pair $(\phi,\eta)\in C(\mathcal{T}_{t_0,T}\times X,\mathbb{T}\times\cM(X))$ is called a \emph{solution} to the IVP  \eqref{characteristic-Eq-1} if it satisfies \eqref{characteristic-Eq-int-1} for all $x\in X$ and 
$t \in \mathcal{T}_{t_0,T}$,
i.e., for some $T>0$. In particular, $(\phi,\eta)$ is called a \emph{global solution} to the IVP  \eqref{characteristic-Eq-1} if
it exists for $t \in [t_0,\infty)$. 
\end{defn}

\subsection{Main results}\label{sec: main results}

In the following, let $X$ be a compact subset of finite-dimensional Euclidean space equipped with the metric $d_X(x,y)$. We use $d_{\mathbb{T}^k}(\phi,\varphi)$ to denote the geodesic distance between $\phi,\varphi$ on the unit circle/torus $\mathbb{T}^k:=\R^k/\Z^k$ for $k=1,2$. We use $|\cdot|$ to denote the usual Euclidean 2-norm. Hence, we trivially have $d_{\mathbb{T}^k}(\phi,\varphi)\le|\phi-\varphi|$. Let $\cM_{+}(X)$ be the space of finite positive measures and $\cM_{\abs}(X)$ be the space of finite measure continuous with respect to the reference measure $\mu_X$. Furthermore, let $\cM_{+,\abs}(X)=\cM_{\abs}(X)\cap \cM_{+}(X)$. For any $a\in\R$ and $T>0$, let $\mathcal{T}_{a,T}=[a,a+T]$. The first main theorem in this paper shows that the IVP for the continuum limit \eqref{characteristic-Eq-1} under the assumptions below is well-posed and yields graphs with positive edge-weights. In other words,  (cf. Section \ref{sec: prelimin}). Let $\cM_{\abs}(X)$ be the set of all positive measures absolutely continuous with respect to $\mu_X$. Assume
\begin{itemize}
\item[(A1)] $D\colon \mathbb{T}^2\to\R$ is Lipschitz continuous\footnote{Equivalently, $D$ can be extended to be a period-1 (coordinate-wise) Lipschitz continuous function on $\R^2$.}: For $\phi,\varphi\in\mathbb{T}^2$,
\begin{equation*}
|D(\phi) - D(\varphi)| \le \Lip(D) d_{\mathbb{T}^2}(\phi,\varphi),
\end{equation*}
where $\Lip(D)$ is the Lipschitz constant of function $D$.
\item[(A2)] $H\colon \mathbb{T}^2\to\R$ is Lipschitz continuous: For $\phi,\varphi\in\mathbb{T}^2$,
\begin{equation*}
|H(\phi) - H(\varphi)| \le \Lip(H)d_{\mathbb{T}^2}(\phi,\varphi).
\end{equation*}
\item[(A3)] $\omega:X \times \mathbb{T}\times \mathbb{R} \to \mathbb{R} $ is continuous in $t$ and Lipschitz continuous in $\phi$ and $x$: For all $x,x'\in X$, $t \in \mathbb{R}$, and $\phi,\varphi \in \mathbb{T}$,
\begin{equation*}
|\omega(x,\phi,t) - \omega(x',\varphi,t)| \le \Lip(\omega)\Big( d_X(x,x') + d_{\mathbb{T}^2}(\phi,\varphi) \Big).
\end{equation*}
\item[(A4)] $(\phi_0,\eta_0)\in C(X,\mathbb{T}\times\cM(X))$.
\item[(A5)] $\eta_0\in C(X,\cM_{+,\abs}(X))$ and $W(x,y):=\frac{\rd \eta^x_0 }{\rd \mu_X}(y)$ satisfies 
\begin{equation}
\label{condition-positivity}
\inf_{x,y\in X}W(x,y)\ge\|H\|_{\infty}(\txte^{\varepsilon T}-1),\end{equation}	
where $\|H\|_{\infty}=\sup_{(\phi,\varphi)\in\mathbb{T}^2}|H(\phi,\varphi)|$ is the supremum norm of the continuous function $H$.
\end{itemize}

We make a few remarks about the above assumptions. (A1)-(A3) are the regularity conditions of the model. The continuity in the initial condition (A4) is crucial in establishing not only the well-posedness but also the approximation result (see Theorem A and Theorem B below). (A5) is a condition for the positivity of the limit graph $\eta_t$ for $t\in\cT_{t_0,T}$. We are now ready to state our first main result.

\begin{theorem}\textbf{(Existence and uniqueness of solutions)} \label{thm A}\\
Assume (A1)-(A4), then there exists a unique global solution $(\phi, \eta) \in C([t_0,+\I)\times X,\mathbb{T}\times\cM(X))$ of the continuum limit equation \eqref{characteristic-Eq-1}. Furthermore, if (A5) holds, then we have $(\phi(t), \eta_t)\in C(X,\mathbb{T}\times\cM_{+}(X))$ for all $t\in\mathcal{T}_{t_0,T}$.
\end{theorem}

The proof of this result is given in Section \ref{sec: ex and un}. Our second main result justifies~\eqref{characteristic-Eq-1} as the continuum limit of \eqref{coevolution}, more precisely, \eqref{characteristic-Eq-1} can be approximated by a sequence of discrete models in the form of~\eqref{characteristic-Eq-1} as $N\to \infty$. To state this result, additional to the assumptions (A1)- (A4), we need to impose additional regularity conditions:

\begin{itemize}
\item[(A6)] $W(x,y)$ defined in (A5) is continuous in both $x$ and $y$.
\item[(A7)] For every $N\in\N$, there exists a uniform partition $(X^N_i)_{i \in [N]}$ of $X$ with respect to $\mu_X$ such that 
$$\mu_X(X^N_i) = \frac{1}{N},\quad i \in [N],\qquad X=\bigcup_{i=1}^N X^N_i$$ and $$\lim_{N\to \infty} \sup_{i \in [N]} diam(X_i^N) = 0.$$
\end{itemize}

To state the discrete approximation result, it is helpful to introduce suitable discretized functions as well as matrices. Let $(X^N_i)_{i \in [N]}$ be a partition satisfying (A7), for $N\in\N$. For every $N \in \mathbb{N}$, $i \in [N]$, let $\omega_i^N$ the frequencies be given by 	
\begin{equation}
\label{eq: omegas}
\omega^N_i(\varphi,t) := N \int_{X_i^N} \omega(y,\varphi,t)~\rd \mu_X(y),\quad \varphi\in \mathbb{T},\quad t\in\R.
\end{equation}
Let the (initial) weights be given by the formula
\begin{equation}
\label{eq: weigths 1}
\quad W^{N}_{i,j} := N^2\int_{X^N_i\times X^N_j} W(x,y) ~\txtd \mu_X(x) ~\txtd \mu_X(y), \quad  i,j \in [N],
\end{equation}
or via the formula
\begin{equation}
\label{eq: weights 2}
W^{N}_{i,j} := W(x_0,y_0), \quad  i,j \in [N],
\end{equation} 
for any  $(x_0,y_0) \in  X^N_i \times X^N_j$. The initial conditions $\phi^{N,0} \in \mathbb{T}^N$ are given by
\begin{equation}
\label{eq: init cond}
\phi^{N,0}_i:= N \int_{X^N_i} \phi_0(y)~\rd y.
\end{equation}	
Let $\{\phi^N_j(t)\}_{1\le j\le N}$ be the solution to \eqref{coevolution} with \eqref{eq: omegas}, \eqref{eq: weigths 1} and \eqref{eq: init cond}.

Finally, we use some appropriate metric $d_{\cT_{t_0,T},\infty}$ (see Equation \eqref{def: metric hatd} in Section \ref{sec: prelimin} for a precise definition) to compare solutions of \eqref{characteristic-Eq-1} with those of \eqref{coevolution}. 
To compare with the solution to \eqref{characteristic-Eq-1},
we define the ``lifted'' tuple $(\phi^N, \eta^N) \in C(\cT_{t_0,T}, \mathcal{B}(X,\mathbb{T})) \times  C(\cT_{t_0,T}, \mathcal{B}(X,\cM(X)))$\footnote{Here $\mathcal{B}(X,\cM(X))$ is the space of bounded measurable functions from $X$ to $\cM(X)$.} from the solution of the discretized model \eqref{characteristic-Eq-1} as follows:
\begin{equation}
\label{eq: weights in intro 1}
\phi^N(t,x):= \sum_{j=1}^N \phi_j^N(t) \chi_{X_j^N}(x),
\end{equation}
where $\xi_B$ is the indicator function of a set $B$ and $\eta^{N} \in C(\cT_{t_0,T}, \mathcal{B}(X,\cM(X)))$ is defined by the
step-function $W^N: \cT_{t_0,T} \times X^2 \to \mathbb{R}$ via
\begin{equation}
\label{eq: step graphon intro}
W^N(t;x,y) := \sum_{i,j = 1}^N W^N_{i,j}(t) \chi_{X_i^N \times X_j^N}(x,y), \quad (x,y) \in X^2;\quad \rd(\eta^{N})_t^x(y):= W^N(t;x,y)\rd \mu_X(y).
\end{equation}	

\begin{theorem} \textbf{(Discrete approximation)} \label{thm: discr approx}\\
Assume (A1)-(A2), and (A4)-(A7). Let $\phi(t,x)$ be the solution to the IVP \eqref{characteristic-Eq-1}. Let 
\benn
\Big( (\phi_j^N)_{j=1}^N,(W_{i,j}^N)_{i,j = 1}^N \Big) \in C(\cT_{t_0,T},~\mathbb{T}^N\times \mathbb{R}^{N\times N}) 
\eenn
be the solution to \eqref{coevolution} with \eqref{eq: omegas}, \eqref{eq: weigths 1} or \eqref{eq: weights 2}, and \eqref{eq: init cond}. Let $(\phi^N, \eta^N) \in C(\cT_{t_0,T}, L^\infty(X,\mathbb{T})) \times  C(\cT_{t_0,T},  \mathcal{B}(X,\cM(X)))$ be defined in \eqref{eq: weights in intro 1} and \eqref{eq: step graphon intro}. Then 
\begin{equation}
\label{eq: disc apr}
\lim_{N \to \infty} d_{\cT_{t_0,T},\infty}\left((\phi,\eta),(\phi^N,\eta^N) \right) = 0,
\end{equation}	
where $d_{\cT_{t_0,T},\infty}$ is defined in~\eqref{def: metric hatd}.
\end{theorem}

\subsection{Outline of the paper}
In Section \ref{sec: prelimin} we will introduce notation and provide some background material from measure theory. After that, Section \ref{sec: ex and un} is devoted to the proof of Theorem \ref{thm A}. In Section \ref{sec: reg} we prove continuous dependence, and hence well-posedness of the limiting integro-differential equation. In Section~\ref{sec: disc apr} we develop the proof of Theorem \ref{thm: discr approx}. To do this, we need to show that the co-evolutionary Kuramoto-type models of the form \eqref{coevolution} can be viewed as a special continuum limit \eqref{characteristic-Eq-1} (Proposition \ref{prop: ODE as integro-differ eq}). In Section \ref{sec: main application}, we apply our main results to a special Kuramoto model studied in~\cite{Berner2019}. We conclude with a brief discussion of our results and an outlook in Section~\ref{sec: lim of the framework}.

\section{Preliminaries}
\label{sec: prelimin}
As we have described in the introduction (cf.~Section \ref{sec: intro cont lim}), we view a weighted, time-evolving graph on the compact vertex set $X$ as a family of time-dependent measures $ (\eta_t^x )_{x \in X}$, which can be viewed as the disintegration of a family of measures $\eta_t$ on the edge set $X\times X$. For a very detailed explanation and justification, why we can interpret these measures as generalized graphs, we refer the reader to \cite{Scededgy2018,KuehnXu}.

For a set $A\subseteq X$, let $diam A$ be the diameter of $A$. Here, $\cM(X)$ denotes the space of all finite signed measures on $X$ and $\cM_+(X)$ the space of all finite positive measures. Let $\mathcal{B}_1(X)$ be the space of measurable bounded functions $f: X \to \mathbb{R}$ such that $\sup_{x\in X}|f(x)|\le1$ and let $\mathcal{BL}_{1}(X)$ be the space of Lipschitz continuous functions $f: X \to \mathbb{R}$ with Lipschitz constant $\Lip(f) \le 1$. Recall that \cite[Chapter 8]{B07}
\begin{itemize}
\item the space  $\cM(X)$ equipped with the total variation norm 
\[\|\mu\|^*_{\sf TV}\colon=\sup_{f\in\mathcal{B}_1(X)}\int_X f~\rd\mu\]
is a Banach space.
\item the space $\cM_+(X)$ equipped with the bounded Lipschitz norm
\[\|\mu\|^*_{\BL}\colon=\sup_{f\in\mathcal{BL}_1(X)}\int_X f~\rd\mu\] is a Banach space.
\end{itemize}
Next, we provide basic definitions and properties for spaces of measure-valued functions from \cite{KuehnXu}.

\begin{defn}
Let $(\eta_t)_{t\in\R}\subseteq\cM(X)$. If 
$$\lim_{\varepsilon\to 0}\frac{\eta_{t+\varepsilon}-\eta_{t}}{\varepsilon}\in\cM(X)$$ 
exists, then
\[\frac{\rd\eta_t}{\rd t}=\lim_{\varepsilon\to 0}\frac{\eta_{t+\varepsilon}-\eta_{t}}{\varepsilon}\] 
is called the \emph{derivative of $\eta_t$ at $t$}.
\end{defn}

\begin{ex}
\label{exmp: measure derivative}
Let $F\in C^1(X)$ and $\mu\in\cM(X)$. Define 
$$\eta_t=F(t)\mu\in\cM(X),\quad \text{for all}\quad t\in\R.$$ 
Then we claim that $\frac{\rd\eta_t}{\rd t}=F'(t)\mu$. Indeed, for any $f\in C_\textnormal{b}(X)$, we have $f\in L^1(X;\mu)$ since $X$ is compact, which implies by Dominated Convergence Theorem that
\begin{align*}
&\lim_{\varepsilon\to0}\int_{X}f~\rd\frac{F(t+\varepsilon)-F(t)}{\varepsilon}\mu = \lim_{\varepsilon\to0}\int_{X}f\frac{F(t+\varepsilon)-F(t)}{\varepsilon}~\rd\mu\\
=&\int_{X}f\lim_{\varepsilon\to0}\frac{F(t+\varepsilon)-F(t)}{\varepsilon}~\rd\mu 
= \int_{X}fF'(t)~\rd\mu=F'(t)\int_Xf~\rd\mu,
\end{align*}
i.e., we find $\frac{\rd\eta_t}{\rd t}=F'(t)\mu$.
\end{ex}

Not all families of parameterized measures are differentiable.

\begin{ex}
Let $\eta_t=\delta_t$. Then for all $f\in C_\txtb(\R)$ we get
\benn
\lim_{\varepsilon\to0}\int_{\R}f~\rd\frac{\delta_{t+\varepsilon}-\delta_t}{\varepsilon}=\lim_{\varepsilon\to0}\frac{f(t+\varepsilon)-f(t)}{\varepsilon},
\eenn
which may not exist, e.g., consider the function $f(x)=(1+\sqrt{x})^{-1/2}$. 
\end{ex}

\begin{prop}
\label{prop-derivative}
Let $\mathcal{N}$ be a compact interval of $\R$. Assume $\eta\in C(\mathcal{N},\cM(X))$. Let $t_0\in\mathcal{N}$. Then $\xi_t=\int_{t_0}^t\eta_{\tau}~\rd\tau\in\cM(X)$ understood in the weak sense
\[\int_Xf~\rd\xi_t=\int_{t_0}^t\left(\int_Xf~\rd\eta_{\tau}\right)~\rd\tau,\quad \forall f\in C_\txtb(X),\]
is differentiable at $t$ for all $t\in\mathcal{N}$, where the derivative is understood as one-sided for the two endpoints of $\cN$.
\end{prop}

\begin{proof}
Since $\eta\in C(\mathcal{N},\cM(X))$, we have $\|\eta_t\|^*_{\TV}<\infty$ for all $t\in\mathcal{N}$. Hence, $\xi_t\in\cM(X)$ for all $t\in\cN$. Moreover, for $\min\cN<t<\max\cN$, let $0<\varepsilon\le\max\cN-t$, so that we obtain
\begin{align*}
\frac{1}{\varepsilon}\int_Xf~\rd(\xi_{t+\varepsilon}-\xi_t)
=\frac{1}{\varepsilon}\int_{t}^{t+\varepsilon}\left(\int_Xf~\rd\eta_{\tau}\right)\rd\tau,\quad \forall f\in C_\txtb(X),
\end{align*}
which implies that
\begin{align*}
&\left|\frac{1}{\varepsilon}\int_Xf~\rd(\xi_{t+\varepsilon}-\xi_t)-\int_X~\rd\eta_{t}\right| \le\frac{1}{\varepsilon}\int_{t}^{t+\varepsilon}\left|\int_Xf~\rd(\eta_{\tau}-\eta_t)\right|~\rd\tau\\
\le&~\left(\frac{\max_X f-\min_X f}{2}\right)\frac{1}{\varepsilon}\int_{t}^{t+\varepsilon}\|\eta_{\tau}-\eta_t\|_{\TV}^*~\rd\tau
\le\frac{\max_X f-\min_X f}{2}\max_{\tau\in[t,t+\varepsilon]}\|\eta_{\tau}-\eta_t\|_{\TV}^*.
\end{align*}
Analogously, we get
\begin{align*}
\left|\frac{1}{\varepsilon}\int_Xf~\rd(\xi_{t-\varepsilon}-\xi_t)-\int_Xf~\rd\eta_{t}\right|
\le&\frac{\max_X f-\min_X f}{2}\max_{\tau\in[t-\varepsilon,t]}\|\eta_{\tau}-\eta_t\|_{\TV}^*.
\end{align*}
Since $\eta\in C(\mathcal{N},\cM(X))$ is uniformly continuous in $t\in\cN$ and $\frac{\max_X f-\min_X f}{2}<\infty$, we have
\[\lim_{\varepsilon\to0}\frac{1}{\varepsilon}\int_Xf~\rd(\xi_{t-\varepsilon}-\xi_t)=\int_Xf~\rd\eta_{t},\quad \forall f\in C_\txtb(X),\]
i.e., this just means
\begin{equation}\label{diff}
\xi'(t)=\eta_{t}.
\end{equation} 
Similarly, \eqref{diff} holds for $t=\min\cN$ and $t=\max\cN$ upon computing with suitable one-sided limits.
\end{proof}

Recall that $\mathcal{B}(X,\cM(X))$ denotes the space of bounded measurable functions from $X$ to $\cM(X)$. For any $\eta\in \mathcal{B}(X,\cM(X))$, let 
$$\|\eta\|^* :=\sup_{x\in X}\|\eta^x\|^*_{\TV}.$$ 
Let $\eta\in C(\mathcal{T}_{t_0,T}, \mathcal{B}(X,\cM(X)))$, then with a slight abuse of notation, which will be clear from the context, we set
$$\|\eta\|^* :=\sup_{t\in\mathcal{T}_{t_0,T}}\sup_{x\in X}\|\eta^x_t\|^*_{\TV}\ge\sup_{t\in\mathcal{T}_{t_0,T}}\sup_{x\in X}|\eta_t^x|(X).$$
In particular, the upper-star notation will always denote the supremum norm over all free variables. Note that if the cardinality of $X$ is infinite, then the topology induced by the bounded Lipschitz norm is strictly weaker than that induced by the total variation norm, and hence by Banach's Theorem the space $\cM$ equipped with the bounded Lipschitz norm is not complete since the two norms are \emph{not} equivalent \cite{B07}. Furthermore, for any $(\phi,\eta),(\varphi,\xi) \in L^{\infty}(X, \mathbb{T})\times \mathcal{B}(X,\cM(X))$ we define the metric
\begin{equation}
\label{eq: metric tilded}
d_{\infty}((\phi,\eta),(\varphi,\xi)) := d_{\mathbb{T},\infty}(\phi,\varphi)+  \|\eta-\xi\|^*,
\end{equation}
where $d_{\mathbb{T},\infty}(\phi,\varphi):= \textrm{esssup}_{x\in X}d_{\mathbb{T}}(\phi(x), \varphi(x))$. Finally, for any compact interval $\mathcal{N}\subseteq\R$, define a uniform metric on 
$C(\mathcal{N}, L^\infty(X,\mathbb{T})) \times  C(\mathcal{N},  \mathcal{B}(X,\cM(X)))$ by
\begin{equation}
\label{def: metric hatd}
d_{\mathcal{N},\infty}((\phi,\eta),(\varphi,\xi))=\sup_{t\in\mathcal{N}} d_{\infty}((\phi(t),\eta_t),(\varphi(t),\xi_t)).
\end{equation}

\begin{prop}
\label{prop: complete metr spa}
Let $\mathcal{N}\subseteq\R$ be a compact interval. Then 
\begin{itemize}
\item  $C(X,\mathbb{T}\times\cM(X))$ endowed with the metric $d_{\infty}$ defined in \eqref{eq: metric tilded} is a complete metric space.
\item  $C(\mathcal{N}\times X,\mathbb{T}\times\cM(X))$ endowed with the metric $d_{\mathcal{N},\infty}$ defined in \eqref{def: metric hatd} is a complete metric space. 
\end{itemize}
\end{prop}

\begin{proof} 
Recall that $(\mathbb{T},d_{\mathbb{T}})$ is a complete metric space and $(\cM(X),\|\cdot\|_{\TV})$ is a Banach space \cite{B07}. Since the space of continuous functions from one complete metric space to another complete metric space is also complete, the conclusions follow.
\end{proof}

\section{Existence and uniqueness of solutions}
\label{sec: ex and un}

\begin{proof}[Proof of Theorem \ref{thm A}]
For a small $ 0<t_*$ (to be specified later) we set $\mathcal{T}=\mathcal{T}_{t_0,t_*} = [t_0, t_0+t_*]$. In the following, we prove the conclusions in several steps. We first show the solution exists locally in a subset of $C(\mathcal{T}\times X,\mathbb{T}\times\cM(X))$, and then we show the uniqueness of solutions in $C(\mathcal{T}\times X,\mathbb{T}\times\cM(X))$ using Gronwall's  inequality. Next, we extend the solution to an open maximal existence interval, and using an a-priori estimate we can show the global existence. Finally we prove the positivity of the second component of the solution.

To show the local existence of solutions, we will construct a subspace $\Omega$ of $C(\mathcal{T} \times X,\mathbb{T}\times\cM(X))$ and apply the Banach fixed point theorem to the space $\Omega$ of solutions. Note that, due to the compactness of $\mathbb{T}^2$ and the continuity of $H,D$, we have that the functions $H,D$ are uniformly bounded. In the following we assume $0<t_*<\varepsilon^{-1}$. Let $\sigma\ge\frac{\|H\|_{\infty}+\|\eta_0\|^*}{1-t_*\varepsilon}$, where we recall $\|\eta_0\|^*=\sup_{x\in X}\|\eta_0^x\|^*_{\TV}$, and the convention
$$\|\eta-\xi\|^*=\sup_{t\in \mathcal{T}}\sup_{x\in X}\|\eta_t^x-\xi_t^x\|^*_{\TV},\quad \eta,\xi\in C(\mathcal{T}\times X,\cM(X)).$$
We are ready to define the space 
$$\Omega=\{(\phi,\eta)\in C(\mathcal{T}\times X,\mathbb{T}\times\cM(X))\colon \phi(t_0,x)=\phi_0(x),\ \eta_{t_0}^x=\eta_0^x,\quad \forall x\in X,\quad \|\eta-\eta_0\|^*_{\TV}\le\sigma\},$$ 
where $\|\eta-\eta_0\|^*_{\TV}=\sup_{t\in\mathcal{T}}\sup_{x\in X}\|\eta_t^x-\eta_0^x\|^*_{\TV}$. Henceforth, we also use $\eta_0\in C(\mathcal{T}\times X,\cM(X))$ for the constant function $(\eta_0)^x_t=\eta_0^x$ for $t\in\mathcal{T}$ and $x\in X$. Note that $\Omega$, equipped with the metric $\widehat{d}$ defined in \eqref{def: metric hatd}, is complete since it is a closed subset of the complete metric space $C(\mathcal{T}\times X,\mathbb{T} \times \cM(X))$, cf.~Proposition \ref{prop: complete metr spa}. Next, we define the operator $\mathcal{A}=(\mathcal{A}^{1},\mathcal{A}^{2})=\{\mathcal{A}^{x}\}_{x\in X}=\{(\mathcal{A}^{1,x},\mathcal{A}^{2,x})\}_{x\in X}$ from $\Omega$ to $\Omega$: For every $x\in X$, and $(\phi,\eta)\in\Omega$,
\begin{alignat}{2}
\mathcal{A}^{1,x}(\phi,\eta)(t)=&~\phi_0(x)+\int_{t_0}^t\left[\omega(x, \phi(x,\tau),\tau)+\int_XD(\phi(\tau,x),\phi(\tau,y))
\rd \eta_{\tau}^x(y)
\right]\rd\tau\label{characteristic-1}\\
 (\mathcal{A}^{2,x}(\phi,\eta)(t))(y)=&~\eta_{0}^x	(y)-\varepsilon\int_{t_0}^t\eta_{\tau}^x(y)\rd\tau
-\varepsilon\left(\int_{t_0}^tH(\phi(\tau,x),\phi(\tau,y))\rd\tau\right)\mu_X(y),\quad y\in X \label{characteristic-2}.
\end{alignat}
Here we note that the definition of the operator $\mathcal{A}^{2,x}$ is to be understood as in  \eqref{eq: weak form of eq. for meas}, i.e., in the weak sense. 
In Steps~1. and 2. below, we  show that the $n$-th iteration $\mathcal{A}^n$ for some $n\in\N$ is a contraction mapping from $\Omega$ to $\Omega$.

\begin{enumerate}
\item[Step~1.] $\mathcal{A}$ is a mapping from $\Omega$ to $\Omega$, this means we need to show: For $(\phi, \eta) \in \Omega$, we also have $\mathcal{A}(\phi, \eta) \in \Omega$.

\begin{enumerate}
	\item[Step~1.1.] It is obvious that $\mathcal{A}^{1,x}(\phi,\eta)(t)\in\mathbb{T}$.
	That $\mathcal{A}^{2,x}(\phi,\eta)(t)\in\cM(X)$ for $t\in\mathcal{T}$ follows from
	\[\sup_{\tau\in\mathcal{T}}\sup_{y\in X}\left|H(\phi(\tau,x),\phi(\tau,y))\right|\le\|H\|_{\infty}\]as well as
	\[\sup_{t\in\mathcal{T}}|\eta_t^x|(X)\le\|\eta_0\|^*+\sigma<\infty.\]
	Furthermore, it is trivial to check that $\mathcal{A}(\phi,\eta)^{1,x}(t_0) = \phi_0(x)$ and
	$\mathcal{A}(\phi,\eta)^{2,x}(t_0) = \eta_0^x$
	\item[Step~1.2.] We want to show that $t\mapsto \mathcal{A}^{x}(\phi,\eta)(t)$ is continuous. Let $t,\ t'\in\mathcal{T}$ with $t<t'$.
	\begin{align*}
	&|\mathcal{A}^{1,x}(\phi,\eta)(t)-\mathcal{A}^{1,x}(\phi,\eta)(t')|\\
	=&\left|\int_t^{t'}\left[\omega(x, \phi(x,\tau),\tau)+\int_XD(\phi(\tau,x),\phi(\tau,y))
	\rd\eta_{\tau}^x(y)
	\right]\rd\tau\right|\\
	\le&\int_t^{t'}\left|\omega(x, \phi(x,\tau),\tau)+\int_XD(\phi(\tau,x),\phi(\tau,y))
	\rd\eta_{\tau}^x(y)
	\right|\rd\tau\\
	\le&\left(\| \omega \|_\infty +\|D\|_{\infty}\sup_{t\in\cN}|\eta_t^x|(X)\right)|t-t'|\\
	\le&\left(\| \omega \|_\infty+\|D\|_{\infty}(\|\eta_0\|^*+\sigma)\right)|t-t'|\to0,\quad \text{as}\ |t-t'|\to0.
	\end{align*}
	Next, we verify the continuity of $t\mapsto\mathcal{A}^{2,x}(\phi,\eta)(t)$. By the definition of total variation norm, for $t<t'$, and the definition of $\mathcal{A}^{2,x}(\phi,\eta)(t)$ in \eqref{characteristic-2} (which, once again we recall, is to be understood, as \eqref{eq: weak form of eq. for meas}, in the weak sense)
	\begin{align*}
	&\|\mathcal{A}^{2,x}(\phi,\eta)(t)-\mathcal{A}^{2,x}(\phi,\eta)(t')\|_{\TV}^*\\
	=&\sup_{f\in\mathcal{B}_1(X)}\int_Xf~\rd\left(\mathcal{A}^{2,x}(\phi,\eta)(t)-\mathcal{A}^{2,x}(\phi,\eta)(t')\right)\\
	\le&~\varepsilon\int_t^{t'}\sup_{f\in\mathcal{B}_1(X)}\int_Xf(y)~\rd\left(\eta_{\tau}^x(y)-H(\phi(\tau,x), \phi(\tau,y))
	~\rd\mu_X(y)\right)~\rd\tau\\
	\le&~\varepsilon|t'-t|\left(\sup_{\tau\in\cT}\|\eta_0^x-\eta_{\tau}^x\|^*_{\TV}+\sup_{\tau\in\cT}\|\eta_0^x(\cdot)-
	H(\phi(\tau,x),\phi(\tau,\cdot))\mu_X(\cdot) \|_{\TV}^*\right)\\
	\le&~\varepsilon|t'-t|(\sigma+2\|\eta_0\|^*+\|H\|_{\infty})\to0,\quad \text{as}\quad |t-t'|\to0.
	\end{align*}
	Here in the last inequality we used that
	\begin{align*}
	&~\sup_{\tau\in\cT}\|\eta_0^x(\cdot)-
	H(\phi(\tau,x),\phi(\tau,\cdot))\mu_X(\cdot) \|_{\TV}^*\\
	&~\le \|\eta_0\|^* + 	\sup_{\tau\in\cT} \sup_{f \in \mathcal{B}_1(X)}\int_X \underbrace{ f(y)H(\phi(\tau,x),\phi(\tau,y)) }_{ \le \| H \|_\infty}~\rd\mu_X(y) \\
	&~\le \|\eta_0\|^*+\|H\|_{\infty}.
	\end{align*}

	\item[Step~1.3.] We show $x\mapsto\mathcal{A}^{x}(\phi,\eta)(t)$ is continuous provided $x\mapsto\eta_0^x$ is continuous. First, we verify the continuity of $\mathcal{A}^{1,x}(\phi,\eta)(t)$ in $x$. For $x,x' \in X, t \in \cT$ we have
		\begin{align*}
		|\mathcal{A}^{1,x}(\phi,\eta)(t)-\mathcal{A}^{1,x'}(\phi,\eta)(t) |
		& \leq  | \phi_0(x) - \phi_0(x')| \\
		& \quad + \int_{t_0}^t |\omega (x',\phi(x,\tau),\tau) - \omega (x',\phi(x',\tau),\tau)| ~\rd\tau\\
		& \quad + \Big| \int_{t_0}^t D(\phi(\tau,x), \phi(\tau,y)) \rd \Big( \eta_{\tau}^x -\eta_{\tau}^{x'} \Big) (y) \Big|\\
		& \quad + \Big| \int_{t_0}^t D(\phi(\tau,x), \phi(\tau,y)) - D(\phi(\tau,x'), \phi(\tau,y))~\rd\eta_{\tau}^{x'} (y) \Big| \\
		&\leq 	 | \phi_0(x) - \phi_0(x')| \\
		& \quad + (t- t_0)\Lip(\omega) d_X(x,x') + \Lip(\omega)\int_{t_0}^t |\phi(x,\tau) -\phi(x',\tau)|~ \rd\tau\\
		& \quad + \| D \|_\infty \int_{t_0}^t\|\eta_{\tau}^x-\eta_{\tau}^{x'}\|_{\TV}^*~\rd\tau\\
		& \quad + \Lip(D) \int_{t_0}^t |\phi(x,\tau) -\phi(x',\tau)| \underbrace{\|\eta_{\tau}^{x'}\|^*_{\TV}}_{\leq \| \eta_0 \|^* + \sigma}~ \rd\tau\\	
		& \quad \to 0,\quad \text{as}\quad d_X(x,x')\to0,
		\end{align*}
		due to the Dominated Convergence Theorem,
		since $\phi_0\in C(X,\mathbb{T})$, as well as $\|\eta_{\tau}^x-\eta_{\tau}^{x'}\|_{\TV}^*\to0$, $|\phi(\tau,x)-\phi(\tau,x')|\to0$, as $d_X(x,x')\to0$ for every $\tau\in\cT$.
	
	Next, we verify the continuity of $x\mapsto\mathcal{A}^{2,x}(\phi,\eta)(t)$. By the definition of total variation norm, for $x,x'\in X,  t\in \cT$, and the definition of $\mathcal{A}^{2,x}(\phi,\eta)(t)$ in \eqref{characteristic-2} we have 
	\begin{align*}
	&\|\mathcal{A}^{2,x}(\phi,\eta)(t)-\mathcal{A}^{2,x'}(\phi,\eta)(t)\|_{\TV}^*\\
	=&~\sup_{f\in\mathcal{B}_1(X)}\int_Xf~\rd\left(\mathcal{A}^{2,x}(\phi,\eta)(t)-\mathcal{A}^{2,x'}(\phi,\eta)(t)\right)\\
	\le&~\varepsilon \|\eta_0^x-\eta_0^{x'}\|^*_{\TV}+\varepsilon\int_{t_0}^{t}\|\eta_{\tau}^x-\eta_{\tau}^{x'}\|^*_{\TV}~\rd\tau\\
	&+\varepsilon\int_{t_0}^{t}\int_X|H(\phi(\tau,x),\phi(\tau,y))-H(\phi(\tau,x'),\phi(\tau,y))|~\rd\mu_X(y)~\rd\tau\\
	\le&~\varepsilon \|\eta_0^x-\eta_0^{x'}\|^*_{\TV}+\varepsilon\int_{t_0}^{t}\|\eta_{\tau}^x-\eta_{\tau}^{x'}\|_{\TV}^*~\rd\tau\\
	&+\varepsilon\Lip(H)\int_{t_0}^{t}|\phi(\tau,x)-\phi(\tau,x')|~\rd\tau\to0,\quad \text{as}\quad d_X(x,x')\to0,
	\end{align*}
	due to the Dominated Convergence Theorem, since $\eta_0\in C(X,\cM(X))$, as well as $\|\eta_{\tau}^x-\eta_{\tau}^{x'}\|_{\TV}^*\to0$, $|\phi(\tau,x)-\phi(\tau,x')|\to0$, as $d_X(x,x')\to 0$ for every $\tau\in\cT$.
	
	\item[Step~1.4.] In the last part of the first step, we also have to show that $$\|\mathcal{A}^{2}(\phi,\eta)-\eta_0\|^*=\sup_{t\in\cT}\sup_{x\in X}\|\mathcal{A}^{2,x}(\phi,\eta)(t)-\eta_0^x\|_{\TV}^*\le\sigma.$$ 
	Indeed, since $\mathcal{A}^{2,x}(\phi,\eta)(0)=\eta_0^x$, by Step~1.2.,
	\begin{align*} 
	\|\mathcal{A}^{2,x}(\phi,\eta)(t)-\eta_0^x\|^*_{\TV}\le&~\varepsilon|t-t_0|(\sigma+\|\eta_0\|^*+\|H\|_{\infty})
	\le\varepsilon t_*(\sigma+\|\eta_0\|^*+\|H\|_{\infty})\le\sigma.
	\end{align*}
\end{enumerate}

\item[Step~2.] Next, we want to show that $\mathcal{A}^n$ is a contraction mapping for some $n\in\N$. So we consider $(\phi,\eta),\ (\widetilde{\phi},\widetilde{\eta})\in\Omega$. We estimate the first component of the operator applied to the difference as follows
\begin{align*}
&\left|\mathcal{A}^{1,x}(\phi,\eta)(t)-\mathcal{A}^{1,x}(\widetilde{\phi},\widetilde{\eta})(t)\right|\\
\le&	\left|\int_{t_0}^t \omega(x, \phi(\tau,x),\tau) - \omega(x, \widetilde{\phi}(\tau,x),\tau) ~\rd\tau \right| \\
&+	\left|\int_{t_0}^t\int_X(D(\phi(\tau,x),\phi(\tau,y))
-D(\widetilde{\phi}(\tau,x),\widetilde{\phi}(\tau,y)))
~\rd\eta_0^x(y)~\rd\tau\right|\\ &+\left|\int_{t_0}^t\int_X(D(\phi(\tau,x),\phi(\tau,y))
-D(\widetilde{\phi}(\tau,x),\widetilde{\phi}(\tau,y)))
~\rd(\eta_{\tau}^x(y)-\eta_0^x(y))~\rd\tau\right|\\
&+\left|\int_{t_0}^t\int_XD(\widetilde{\phi}(\tau,x)\widetilde{\phi}(\tau,y))
~\rd\left(\eta_{\tau}^x(y)-\widetilde{\eta}_{\tau}^x(y)\right)~\rd\tau\right|\\
\le& \Lip(\omega) \int_{t_0}^t \|\phi(\tau,\cdot)-\widetilde{\phi}(\tau,\cdot)\|_\infty~\rd\tau \\
&+2\Lip(D)|\eta_0^x|(X)\int_{t_0}^t \|\phi(\tau,\cdot)-\widetilde{\phi}(\tau,\cdot)\|_\infty~\rd\tau\\
&+2\Lip(D)\sup_{\tau\in\cN}\|\eta_{\tau}^x - \eta_0^x\|_{\TV}^*\int_{t_0}^t
\|\phi(\tau,\cdot)-\widetilde{\phi}(\tau,\cdot)\|_\infty~\rd\tau\\ &+\|D\|_{\infty}\left|\int_{t_0}^td_{\TV}(\eta_{\tau}^x,\widetilde{\eta}_{\tau}^x)~\rd\tau\right|.
\end{align*}
Moreover, we have
\begin{align*}
&\|\mathcal{A}^{2,x}(\phi,\eta)(t)-\mathcal{A}^{2,x}(\widetilde{\phi},\widetilde{\eta})(t)\|^*_{\TV}
=\sup_{f\in\mathcal{B}_1(X)}\int_Xf~\rd\left(\mathcal{A}^{2,x}(\phi,\eta)(t)-
\mathcal{A}^{2,x}(\widetilde{\phi},\widetilde{\eta})(t)\right)\\
\le&~\varepsilon\left(\left|\int_{t_0}^t\|\eta_{\tau}^x-\widetilde{\eta}_{\tau}^x\|^*_{\TV}~\rd\tau\right|
+\left|\int_{t_0}^t\int_X|H(\phi(\tau,x),\phi(\tau,y))
-H(\widetilde{\phi}(\tau,x),\widetilde{\phi}(\tau,y))|~\rd\tau~\rd \mu_X(y)\right|\right)\\
\le&~\varepsilon\left(\left|\int_{t_0}^t\|\eta_{\tau}^x-\widetilde{\eta}_{\tau}^x\|_{\TV}^*~\rd\tau\right|+
2\Lip(H)\left|\int_{t_0}^t\|\phi(\tau,\cdot)-\widetilde{\phi}(\tau,\cdot)\|_\infty~\rd\tau\right|\right).
\end{align*}
Combining the two previous estimates and using the triangle inequality we finally get (recall the definition of the metric $d_{\infty}$ in \eqref{eq: metric tilded})
\begin{alignat*}
{2}
d_{\infty}\left(\mathcal{A}(\phi,\eta)(t),\mathcal{A}(\widetilde{\phi},\widetilde{\eta})(t)\right)
\le M_3\left|\int_{t_0}^t d_{\infty}((\phi(\tau),\eta_\tau),(\widetilde{\phi}(\tau),\widetilde{\eta}_\tau))~\rd\tau\right|,
\end{alignat*}
with the constant $M_3 := 4 \Big(\Lip(\omega) + \Lip(D) + \|D\|_\infty + \varepsilon(1+ \Lip(H)) \Big)(\|\eta_0\|^*+\sigma)$.
Taking the supremum over $t \in  \mathcal{T}$ in both sides of last inequality we obtain
\begin{equation}\label{inequality-contraction}
d_{\mathcal{T},\infty}\left(\mathcal{A}(\phi,\eta),\mathcal{A}(\widetilde{\phi},\widetilde{\eta})\right)
\le M_3 t_* d_{\mathcal{T},\infty}((\phi,\eta),(\widetilde{\phi},\widetilde{\eta})).
\end{equation}
Hence for small enough $t_*$ (such that $M_3 t_*<1$) we have that $\mathcal{A}$ is a contraction. By the Banach contraction mapping principle, which is applicable due to Proposition \ref{prop: complete metr spa} and \eqref{inequality-contraction}, there exists a unique fixed point $P$ of $\mathcal{A}$. Hence, there exists a unique solution $(\phi, \eta)$ for $t\in\cT$ in $\Omega\subseteq C(\cT\times X,\mathbb{T}\times\cM(X))$ to the continuum limit equation \eqref{characteristic-Eq-1}. 
\item[Step~3.]
By the previous steps we have established a unique solution in $[t_0,t_0+t_*]$.
We now want to show that a global solution exists  on $[t_0,+\infty)$.
By Zorn's lemma (transfinite induction), one can always extend the solution by repeating Steps~1-2 indefinitely up to a maximal existence time $T_{\max}$ with the dichotomy:
\begin{enumerate}
	\item[(i)] $\lim_{t\uparrow T_{\max}}|\phi(t)|+\|\eta\|^*=\infty$;
	\item[(ii)] $T_{\max}=+\infty$.
\end{enumerate}
Note that $|\phi(t)|\le1$ since $\phi(t)\in\mathbb{T}$.
Moreover, by  \eqref{characteristic-Eq-1b},
\begin{align*}
\|\eta \|^*\le& \|\eta_0\|^*+\|H\|_{\infty}\varepsilon\int_{t_0}^t\txte^{-\varepsilon(t-\tau)}\rd\tau
\le\|\eta_0\|^*+\|H\|_{\infty},\quad t\ge t_0,
\end{align*}
which implies case (i) will never occur. Hence $T_{\max}=+\infty$.
\item[Step~4.]
By Steps 1-3, there exists a global solution $(\phi(t), \eta_t) \in C([t_0,+\infty)\times X,\mathbb{T}\times\cM(X))$ of the continuum limit equation \eqref{characteristic-Eq-1}.
Next, we show that the solution is unique in $C(\cT_{t_0,T}\times X,\mathbb{T}\times\cM(X))$, for any $T>0$. Let $(\phi,\eta),\ (\widetilde{\phi},\widetilde{\eta})\in C(\cT_{t_0,T}\times X,\mathbb{T}\times\cM(X))$ be two solutions to the IVP of \eqref{characteristic-Eq-1}. Similar as in \eqref{inequality-contraction}, one can show: For $t \in \mathcal{T}_{t_0,T}$,
\begin{align*}
d_{\infty}((\phi(t),\eta_t),(\widetilde{\phi}(t),\widetilde{\eta}_t))\le M_3\int_{t_0}^t d_{\infty}((\phi(\tau),\eta_{\tau}),(\widetilde{\phi}(\tau),\widetilde{\eta}_{\tau}))~\rd\tau
\end{align*}
which implies by Gronwall inequality that
\begin{equation*}
d_{\infty}((\phi(t),\eta_t),(\widetilde{\phi}(t),\widetilde{\eta}_t))=0.
\end{equation*}
This shows that the solution to the IVP of \eqref{characteristic-Eq-1} is unique and in $C(\cT_{t_0,T}\times X,\mathbb{T}\times\cM(X))$.

\item[Step~5.] By Steps 1-4, there exists a unique global solution $(\phi(t), \eta_t) \in C([t_0,+\infty)\times X,\mathbb{T}\times\cM(X))$ of the continuum limit equation \eqref{characteristic-Eq-1}. It only remains to prove the last statement about the positivity of $\eta_t$. Since $\eta$ solves \eqref{characteristic-Eq-1b},  for every $x\in X$, we calculate that
\begin{align*}
\eta_t^x(y)=&~\txte^{-\varepsilon (t-t_0)}\eta_0^x(y)-\varepsilon\int_{t_0}^t\txte^{-\varepsilon(t-\tau)}\left(H(\phi(\tau,x),\phi(\tau,y))\right) \mu_X(y)~\rd\tau,
\end{align*}
where last equality is to be understood in the weak sense, cf.~\eqref{eq: weak form of eq. for meas}. For any $B\in\mathfrak{B}(X)$\footnote{Here $\mathfrak{B}(X)$ is the set of Borel subsets of $X$.}, let $\chi_B$ denote the characteristic function on the set $B$. For any
$x\in X$  and $t \in \mathcal{T}_{t_0,T}$ we have \footnote{Note that by a standard approximation argument we can replace the test function $f$ by $\chi_B$ in \eqref{eq: weak form of eq. for meas}.}
\begin{align*}
\eta_t^x(B) &= \int_X \chi_B(y) \rd \eta_t^x(y) \\
=&~\txte^{-\varepsilon (t-t_0)}\eta_0^x(B)-\varepsilon\int_{t_0}^t\txte^{-\varepsilon(t-\tau)}\int_B\left(H(\phi(\tau,x),\phi(\tau,y))\right) \rd\mu_X(y)\rd\tau\\
\ge&~\txte^{-\varepsilon (t-t_0)}\eta_0^x(B)-\varepsilon\int_{t_0}^t\txte^{-\varepsilon(t-\tau)}\|H\|_{\infty} \int_B \rd\mu_x(y)\rd\tau\\
\ge&~\txte^{-\varepsilon (t-t_0)}\eta_0^x(B)-\varepsilon\|H\|_{\infty}\mu_X(B)\int_{t_0}^t\txte^{-\varepsilon(t-\tau)}\rd\tau\\
=&~\txte^{-\varepsilon (t-t_0)}\eta_0^x(B)-\|H\|_{\infty}\mu_X(B)(1-\txte^{-\varepsilon (t-t_0)})>0,
\end{align*}
provided 
$$\inf_{y\in X}\frac{\rd\eta_0^x(y)}{\rd\mu_X(y)}\ge\sup_{t\in\cT_{t_0,T}}\|H\|_{\infty}(\txte^{\varepsilon (t -t_0)}-1)=\|H\|_{\infty}(\txte^{\varepsilon T }-1),$$ 
for all $x\in X$.
\end{enumerate}
\end{proof}

\section{Continuous dependence}
\label{sec: reg}

Having established existence and uniqueness of solutions of the integro-differential equations \eqref{characteristic-Eq-1}, we next want to prove that such equations behave regularly in the sense that they depend continuously on the initial data and on the function $\omega$. Recall (cf.~equation \eqref{eq: metric tilded}) that for any $(\phi,\eta),(\varphi,\xi) \in C(X, \mathbb{T} \times \cM(X))$,
\begin{equation*}
d_{\infty}((\phi,\eta),(\varphi,\xi)) :=  \| \phi - \varphi \|_\infty +  \|\eta-\zeta\|^*.
\end{equation*}

\begin{prop}
\label{prop: reg}
\textbf{(i) (Continuous dependence on the initial conditions)} \\
Let	$(\phi_0,\eta_0),(\widetilde{\phi}_0,\widetilde{\eta}_0)\in C(X,\mathbb{T}\times\cM(X))$. Let $(\phi,\eta),\ (\widetilde{\phi},\widetilde{\eta})\in C(\mathcal{T}_{t_0,T}\times X,\mathbb{T}\times\cM(X))$ be solutions to \eqref{characteristic-Eq-1a} and \eqref{characteristic-Eq-1b} subject to the initial conditions $(\phi_0,\eta_0),(\widetilde{\phi}_0,\widetilde{\eta}_0)$, respectively. Then we have
\begin{equation*}
d_{\infty}((\phi(t),\eta(t)),(\widetilde{\phi}(t),\widetilde{\eta}(t))) \le  d_{\infty}((\phi_0,\eta_0),(\widetilde{\phi}_0,\widetilde{\eta}_0)) \txte^{C_1 t},\quad t \in \mathcal{T}_{t_0,T}.
\end{equation*}
\textbf{(ii) (Continuous dependence on $\omega$)} \\
Let	$(\phi_0,\eta_0)\in C(X,\mathbb{T}\times\cM(X))$ and $\omega, \tilde{\omega} \in C(X)$. Let $(\phi,\eta),\ (\widetilde{\phi},\widetilde{\eta})\in C(\mathcal{T}_{t_0,T}\times X,\mathbb{T}\times\cM(X))$ be solutions to \eqref{characteristic-Eq-1} and \eqref{characteristic-Eq-1} with $\omega$ replaced by $\tilde{\omega}$, respectively. Then we obtain
\begin{equation*}
d_{\infty}((\phi(t),\eta_t),(\widetilde{\phi}(t),\widetilde{\eta}_t)) \le \Big(T \| \omega - \tilde{\omega} \|_\infty \Big) \txte^{C_1t},\quad t \in \mathcal{T}_{t_0,T}.
\end{equation*}
\end{prop}

\begin{proof}
\textbf{(i)} We calculate for any $t \in \mathcal{T}_{t_0,T}$ and $x \in X$ that

\begin{align*}
\left|\phi(t,x)-\widetilde{\phi}(t,x)\right|
&\le
\| \phi_0 - \widetilde{\phi}_0 \|_\infty \\
&\quad+ \left| \int_{t_0}^t \omega(x, \phi(\tau,x),\tau) - \omega(x, \widetilde{\phi}(\tau,x),\tau)  ~\rd\tau \right| \\
&\quad+\left|\int_{t_0}^t\int_X(D(\phi(\tau,x),\phi(\tau,y))-D(\widetilde{\phi}(\tau,x),\widetilde{\phi}(\tau,y)))~\rd\eta_\tau^x(y)~\rd\tau\right|\\
&\quad +\left|\int_{t_0}^t\int_XD(\widetilde{\phi}(\tau,x),\widetilde{\phi}(\tau,y))~\rd\left(\eta_{\tau}^x(y)-\widetilde{\eta}_{\tau}^x(y)\right)~\rd\tau\right|\\
&\le  \|\phi_0 - \widetilde{\phi}_0\|_\infty \\
&\quad + \Lip(\omega) \int_{t_0}^t \|\phi(\tau, \cdot) - \widetilde{\phi}(\tau,\cdot ) \|_\infty~\rd\tau \\
& \quad + 2\Lip(D)\int_{t_0}^t \|\phi(\tau, \cdot) - \widetilde{\phi}(\tau,\cdot ) \|_\infty \|\eta_\tau \|^*~\rd\tau \\
&\quad + \|D\|_\infty\int_{t_0}^t \|\eta_{\tau}^x-\widetilde{\eta}_{\tau}^x \|_{\TV}^*~\rd\tau
\end{align*}
Moreover, by \eqref{eq: weak form of eq. for meas} we have
\begin{align*}
\|\eta_t^x-\widetilde{\eta}_t^x\|^*_{\TV}
&=\sup_{f\in\mathcal{B}_1(X)}\int_Xf~\rd\left(\eta_t^x-\widetilde{\eta}_t^x\right)\\
\le&~\|\eta_0-\widetilde{\eta}_0\|^*_{\TV} 
+ \varepsilon\left|\int_{t_0}^t\|\eta_{\tau}^x-\widetilde{\eta}_{\tau}^x\|^*_{\TV}~\rd\tau\right|\\
&+\varepsilon\left|\int_{t_0}^t\int_X|H(\phi(\tau,x),\phi(\tau,y))
-H(\widetilde{\phi}(\tau,x),\widetilde{\phi}(\tau,y))|~\rd\tau~\rd \mu_X(y)\right|\\
\le&~  \|\eta_0-\widetilde{\eta}_0\|^*_{\TV} + \varepsilon\int_{t_0}^t\|\eta_{\tau}^x-\widetilde{\eta}_{\tau}^x\|^*_{\TV}~\rd\tau\\
&+2 \varepsilon\Lip(H) \int_{t_0}^t \|\phi(\tau, \cdot) - \widetilde{\phi}(\tau,\cdot) \|_\infty~ \rd\tau.
\end{align*}
Hence, combining the last two inequalities, we may conclude that
\begin{align*}
d_{\infty}((\phi(t),\eta_t),(\widetilde{\phi}(t),\widetilde{\eta}_t))&\le d_{\infty}((\phi_0,\eta_0),(\widetilde{\phi}_0,\widetilde{\eta}_0)) + C_1 \int_{t_0}^t d_{\infty}((\phi(\tau),\eta_{\tau}),(\widetilde{\phi}(\tau),\widetilde{\eta}_{\tau}) \rd\tau,
\end{align*}
where $C_1:= \Big(\Lip(\omega) + 2\Lip(D)\|\eta \|^{*}  + 2\varepsilon\Lip(H)  + \varepsilon + \| D \|_\infty\Big)$.
Finally, by Gronwall's inequality it follows that
\begin{equation*}
d_{\infty}((\phi(t),\eta_t),(\widetilde{\phi}(t),\widetilde{\eta}_t))\le d_{\infty}((\phi_0,\eta_0),(\widetilde{\phi}_0,\widetilde{\eta}_0)) \txte^{C_1 t}.
\end{equation*}
\textbf{(ii)} We have that
\begin{align*}
\left|\phi(t,x)-\widetilde{\phi}(t,x)\right|\le &
\left| \int_{t_0}^t \omega(x, \phi(\tau,x),\tau) - \widetilde{\omega}(x, \widetilde{\phi}(\tau,x),\tau)  \rd\tau \right| \\ &+\left|\int_{t_0}^t\int_X(D(\phi(\tau,x),\phi(\tau,y))-D(\widetilde{\phi}(\tau,x),\widetilde{\phi}(\tau,y)))
~\rd\eta_\tau^x(y)~\rd\tau\right|\\
&+\left|\int_{t_0}^t\int_XD(\widetilde{\phi}(\tau,x),\widetilde{\phi}(\tau,y))
~\rd\left(\eta_{\tau}^x(y)-\widetilde{\eta}_{\tau}^x(y)\right)~\rd\tau\right|\\
\le& \left| \int_{t_0}^t \omega(x, \phi(\tau,x),\tau) - \omega(x, \widetilde{\phi}(\tau,x),\tau)~\rd\tau \right| \\
&+ \left| \int_{t_0}^t \omega(x, \widetilde{\phi}(\tau,x),\tau) - \widetilde{\omega}(x, \widetilde{\phi}(\tau,x),\tau)  ~\rd\tau \right| \\ &+\left|\int_{t_0}^t\int_X(D(\phi(\tau,x),\phi(\tau,y))-D(\widetilde{\phi}(\tau,x),\widetilde{\phi}(\tau,y)))
~\rd\eta_\tau^x(y)~\rd\tau\right|\\
&+\left|\int_{t_0}^t\int_XD(\widetilde{\phi}(\tau,x),\widetilde{\phi}(\tau,y))
~\rd\left(\eta_{\tau}^x(y)-\widetilde{\eta}_{\tau}^x(y)\right)~\rd\tau\right| \\
\le& \Lip(\omega) \int_{t_0}^t \|\phi(\tau, \cdot) - \widetilde{\phi}(\tau,\cdot) \|_\infty ~\rd\tau \\
&+ T \|\omega - \tilde{\omega}\|_\infty\\
&+\left|\int_{t_0}^t\int_X(D(\phi(\tau,x),\phi(\tau,y))-D(\widetilde{\phi}(\tau,x),\widetilde{\phi}(\tau,y)))
~\rd\eta_\tau^x(y)
~\rd\tau\right|\\
&+\left|\int_{t_0}^t\int_XD(\widetilde{\phi}(\tau,x),\widetilde{\phi}(\tau,y))
~\rd\left(\eta_{\tau}^x(y)-\widetilde{\eta}_{\tau}^x(y)\right)~\rd\tau\right|.
\end{align*}
Similarly, one can obtain estimates for $\|\eta_t^x-\widetilde{\eta}_t^x\|^*_{\TV}$ as in \textbf{(i)}. Moreover, we obtain following inequality:
\begin{align*}
d_{\infty}((\phi(t),\eta_t),(\widetilde{\phi}(t),\widetilde{\eta}_t))&\le T \|\omega - \tilde{\omega} \|_\infty + C_1\int_{t_0}^t d_{\infty}((\phi(\tau),\eta_{\tau}),(\widetilde{\phi}(\tau),\widetilde{\eta}_{\tau}))~\rd\tau
\end{align*} 
and hence it again follows from Gronwall's inequality  that
\begin{equation*}
d_{\infty}((\phi(t),\eta_t),(\widetilde{\phi}(t),\widetilde{\eta}_t))\le \Big(T \|\omega - \tilde{\omega} \|_\infty \Big) \txte^{C_1t}.
\end{equation*}
This concludes the proof.
\end{proof}

From results established in Sections \ref{sec: ex and un} and \ref{sec: reg}, the integro-differential equation \eqref{characteristic-Eq-1} is well-posed. Note carefully, that the well-posedness does only require the conditions (A1)-(A4), and not the absolute continuity condition (A5). In particular, one only requires the continuity of the measure-valued function $x\mapsto \eta_0^x$. This is indeed strictly weaker than absolute continuity with respect to $\mu_X$ as demonstrated in the example below.

\begin{ex} \textbf{(Weighted graphs on the Cantor set)}\\
Let $X=[0,1]$ and $\mathfrak{C} = \bigcap_{n \in \mathbb{N}_0} F_n$ be the middle-third Cantor set, where $F_0:= I^0_1=[0,1]$ and $F_{n}=\cup_{j=1}^{2^n}I^n_j\subsetneq F_{n-1}$ consisting of $2^n$ disjoint closed intervals, for $n \geq 1$. Recall that for every $n\in\N$, $\mu_X(I^n_j\cap \mathfrak{C}) = \frac{1}{2^n}$ for any $j=1,\ldots,2^n$. Hence for every $n\in\N$, every $x\in\mathfrak{C}\setminus\{1\}$, there exists $1\le j=j_x\le 2^n$ such that $x\in I^n_j$. Let $\mu_X$ be the uniform measure over $\mathfrak{C}$ and $F_{\mu_X}$ be its distribution function. Note that $x\mapsto F_{\mu_X}(x)$ is continuous on $X$. Let $\eta^x\in\mathcal{M}_+(X)$ be such that its generalized distribution function is $$F_{\eta^x}(z)\colon=\int_0^z\rd\eta^x=\max\{0,F_{\mu_X}(z)-F_{\mu_X}(x)\},\quad z\in Z.$$ 
Since $\mu_X$ is continuous in distribution but not absolutely continuous, so is $\eta^x$ for all $x<1$. Moreover, we have
$$\rd\eta^x(z)=\rd\mu_X(z),\quad \text{for}\ z>x.$$
Next, we show $x\mapsto\eta^x$ is continuous in the strong topology induced by total variation distance as well. Let $X\ni x_k\to x$ as $k\to\infty$. Hence for every $n\in\N$, there exists $K\in\N$ such that $x_k,x\in I^n_j$ for some $1\le j=j_x\le n$. Let $A\in\mathfrak{B}(X)$. Note that $$\rd(\eta^x(z)-\eta^{x_k}(z))=0,\quad \text{for}\ z\in X\setminus[\min\{x,x_k\},\max\{x,x_k\}].$$ Hence for all $k\ge K$, 
\begin{align*}
|\eta^x(A)-\eta^{x_k}(A)|=&\left|\int_A\rd(\eta^x-\eta^{x_k})\right|
=\left|\int_{A\cap[\min\{x,x_k\},\max\{x,x_k\}]}\rd(\eta^x-\eta^{x_k})\right|\\
=&\left|\int_{A\cap[\min\{x,x_k\},\max\{x,x_k\}]}\rd\eta_{\min\{x,x_k\}}\right|
=\int_{A\cap[\min\{x,x_k\},\max\{x,x_k\}]}\rd\mu_X\\
\le&\int_{[\min\{x,x_k\},\max\{x,x_k\}]}\rd\mu_X
\le\mu_X(I^n_j)\le2^{-n},
\end{align*}
which implies $d_{\TV}(\eta^x,\eta^{x_k})\le\mu_X(I^n_j)\le2^{-n}$, and verifies that
$\lim_{k\to\infty}d_{\TV}(\eta^x,\eta^{x_k})=0$. 	
\end{ex}

In particular, the last example already shows very clearly the advantages of working with (limiting) graphs as measures without requiring a density. Our next goal is to show that the solutions to \eqref{characteristic-Eq-1} can be approximated by those of the discretized models of the form \eqref{coevolution} under additional assumptions as stated in Theorem~\ref{thm: discr approx}.

\section{Discrete approximation}
\label{sec: disc apr}

In order to be able to compare solutions of the ODE models \eqref{coevolution} with that of the continuum counterpart \eqref{characteristic-Eq-1}, we first show that \eqref{coevolution} can be rewritten in the form of \eqref{characteristic-Eq-1}.

\begin{prop}
\label{prop: ODE as integro-differ eq}
\textbf{(i)} Let $\Big( (\phi_j^N)_{j=1}^N,(W_{i,j}^N)_{i,j = 1}^N \Big) \in C(\cT_{t_0,T}, \mathbb{T}^N\times \mathbb{R}^{N\times N})$ be the solution to \eqref{coevolution} with \eqref{eq: omegas}, \eqref{eq: weigths 1} or \eqref{eq: weights 2}, and \eqref{eq: init cond}. Then $(\phi^N, \eta^N) \in C(\cT_{t_0,T}, L^\infty(X,\mathbb{T})) \times  C(\cT_{t_0,T}, \mathcal{B}(X,\cM(X)))$ defined in \eqref{eq: weights in intro 1} and \eqref{eq: step graphon intro}, solves the integro-differential equation \eqref{characteristic-Eq-1} with	
\begin{equation}
\label{eq: inti cond for the ODE cont lim}
\phi_0^N(x) := \sum_{j= 1}^N  \phi_j^{N,0}\chi_{X_j^N}(x), \quad \omega^N(x,u) := \sum_{j= 1}^N \omega_j^N(u)\chi_{X_j^N}(x).
\end{equation}
\textbf{(ii)} Let $\phi^N \colon\begin{cases} \cT_{t_0,T}\times X \to \mathbb{R},\\ (t,x)\mapsto\sum_{j=1}^N \phi^N_j(t) \chi_{X_j^N}(x),\end{cases}$ and~~~~$W^N \colon \begin{cases}
\cT_{t_0,T}\times X^2 \to \mathbb{R},\\
(t,x,y)\mapsto\sum_{i,j = 1}^N W^N_{i,j}(t)\chi_{X_i^N \times X_j^N}(x,y).\end{cases}$ Let
$$\eta\colon \begin{cases}
\cT_{t_0,T}\times X\to\mathcal{M}(X)\\ (t,x)\mapsto\rd(\eta^{N})_t^x(y)\colon= W^N(t;x,y)~\rd \mu_X(y),\quad y\in X.
\end{cases}$$
Assume that 
$(\phi^N, \eta^N) \in C(\cT_{t_0,T}, L^\infty(X,\mathbb{T})) \times  C(\cT_{t_0,T}, \mathcal{B}(X,\cM(X)))$
solves the integro-differential equation \eqref{characteristic-Eq-1} with initial conditions \eqref{eq: inti cond for the ODE cont lim}. Then
$\Big((\phi_j)_{j=1}^N,(W^N_{i,j})_{i,j = 1}^N \Big) \in C(\cT_{t_0,T}, \mathbb{T}^N\times \mathbb{R}^{N\times N})$ is the solution to the IVP \eqref{coevolution}.
\end{prop}

\begin{proof}
We only prove \textbf{(i)}, as \textbf{(ii)} can be shown analogously. It is not difficult to see that $\phi^N\in C(\cT_{t_0,T}, L^\infty(X,\mathbb{T}))$. Next, we deal with  $\eta^N$:
Since for any $t \in  \cT_{t_0,T}$ we have
\begin{align*}
\sup_{x \in X} \| (\eta^{N})_t^x \| ^*_{\TV} &= \sup_{x \in X} \sup_{f \in \cB_1(X)} \int_X f(y) W^N(t;x,y)~\rd\mu_X(y) \leq \sup_{i,j \in [N]} | W^N_{i,j}(t)|,
\end{align*}
we have that $\eta_t \in \mathcal{B}(X,\cM(X))$. Similarly,  for $t, \tilde{t} \in \cT_{t_0,T}$ with $t \to \tilde{t}$ we have that
\begin{align*}
\sup_{x \in X} \| (\eta^{N})_t^x -  (\eta^N)_{\tilde{t}}^x \|^*_{\TV} &= \sup_{x \in X} \sup_{f \in \cB_1(X)} \int_X f(y) \Big( W^N(t;x,y) - W^N(\tilde{t};x,y) \Big)~\rd\mu_X(y)\\
&\leq \sup_{i,j \in [N]} | W^N_{i,j}(t) -  W^N_{i,j}(\tilde{t})| \to 0,
\end{align*}
which implies that $\eta^N \in C(\cT_{t_0,T},\mathcal{B}(X,\cM(X)))$. Now let $k \in [N]$ and $x \in X^N_k$. Since $\mu_X(X^N_k)=\frac{1}{N}$, we have
\begin{align*}
\frac{\partial\phi^N(t,x)}{\partial t} = \dot{\phi}_k^N(t) &= \omega_k^N(\phi_k^N(t),t) +\frac{1}{N}
\sum_{j=1}^N D(\phi_k^N(t), \phi_j^N(t)) W^N_{kj}(t) \\
&= \omega(x,\phi^N(x,t), t) + \int_X D(\phi^N(t,x), \phi^N(t,y)) W^N(t;x,y) ~\rd\mu_X(y) \\
&= \omega(x,\phi^N(x,t), t) + \int_X D(\phi^N(t,x), \phi^N(t,y)) ~\rd(\eta^{N})_t^x(y).
\end{align*}
Further, similar to Example \ref{exmp: measure derivative} we calculate that
\begin{align*}
\frac{\partial  (\eta^{N})_t^x(y) }{\partial t} &=  \sum_{j=1}^N \dot{W}^N_{kj} (t) \chi_{X^N_j}(y) \mu_X(y) \\
&= - \varepsilon \sum_{j=1}^N  \Big(W^N_{kj}(t) \chi_{X^N_j}(y) \mu_X(y) + H(\phi_k^N(t), \phi_j^N(t))\chi_{X^N_j}(y) \mu_X(y) \Big) \\
&= - \varepsilon \Big( (\eta^N)_t^x(y) + H(\phi^N(t,x), \phi^N(t,y))\mu_X(y)\Big).
\end{align*}
Thus, the tuple $(\phi^N, \eta^N) \in C(\cT_{t_0,T}, L^\infty(X,\mathbb{T})) \times  C(\cT_{t_0,T}, \mathcal{B}(X,\cM(X)))$	solves the IVP \eqref{characteristic-Eq-1}.
\end{proof}

To show Theorem~\ref{thm: discr approx}, we need two more lemmata:

\begin{lem}
\label{lem: disc approx of the intial cond}
Under the notation and assumptions of Theorem \ref{thm: discr approx} we have: \\
\textbf{(i)}
\begin{equation*}
\lim_{N\to \infty} \| W^N(t_0; \cdot) - W \|_\infty = 0.
\end{equation*}
\textbf{(ii)}
\begin{equation*}
\lim_{N\to \infty} \| \eta^N_{t_0} - \eta_0 \|_{\sf TV}^* = 0.
\end{equation*}	
\end{lem}
\begin{proof}
\textbf{(i)} Let $\sigma>0$ be arbitrary. By (A6), since $X^2$ is compact, we have $W$ is uniformly continuous. Hence there exists a $\sigma_1>0$ such that
\begin{equation*}
| W(x,y)  - W(\hat{x}, \hat{y})| <\sigma,\quad \text{whenever}\quad |(x,y)-(\hat{x},\hat{y})|<\sigma_1.
\end{equation*}
Due to (A7) we can find an $N_0 \in \mathbb{N}, N_0= N_0(\sigma),$ such that
\begin{equation*}
\sup_{N \geq N_0}  \sup_{i \in [N]} diam(X^N_i) <\sigma_1.
\end{equation*}
By the definition of $W^N$,
\begin{equation*}
\|	W^N(t_0 ;\cdot) - W \|_\infty  <\sigma.
\end{equation*}
Since $\sigma$ was arbitrary, the conclusion follows.\\
\textbf{(ii)} It follows immediately from the definition of $\|\cdot \|_{\sf TV}^*$ and \textbf{(i)}, since for $N\to \infty$ we have
\begin{equation*}
\begin{split}
\|\eta^N_{t_0} - \eta_0 \|_{\TV}^* =& \sup_{x \in X} \sup_{f\in \mathcal{B}_1(X)}\int f(y) \Big( W^N(t_0;x,y)  - W(x,y) \Big)~\rd\mu_X(y) \\
\le& \sup_{x\in X}\| W^N(t_0;x,\cdot)  - W(x,\cdot) \|_{L^1(X;\mu_X)}\le\|W^N(t_0;\cdot)-W\|_{\infty} \to 0,
\end{split}
\end{equation*}
which finishes the proof.
\end{proof}

\begin{lem}
\label{lem: cont depend 2}
Under the notation and assumptions of Theorem \ref{thm: discr approx} we have:
\begin{equation*}
d_{\cT_{t_0,T},\infty}\left((\phi,\eta),(\phi^N,\eta^N)\right) \le \Big( \| \phi_0 - \phi_0^N \|_\infty  +   \|\eta_0-\eta_{t_0}^N\|^*_{\TV}  +	T \|\omega - \omega^N \|_\infty \Big) \txte^{C_1T},
\end{equation*}
with the constant $C_1:= \Big(\Lip(\omega) + 2\Lip(D)\|\eta \|^{*}  + 2\varepsilon\Lip(H)  + \varepsilon + \| D \|_\infty\Big)$.
\end{lem}
\begin{proof}
By Proposition \ref{prop: ODE as integro-differ eq} we know that both  $(\phi^N, \eta^N)$ and $(\phi,\eta)$ solve the integro-differential equation \eqref{characteristic-Eq-1}. Thus, the desired estimate is obtained with exact the same calculations as in  Proposition \ref{prop: reg} and the triangle inequality.
\end{proof}

\begin{proof}[Proof of Theorem \ref{thm: discr approx}]
From Lemma \ref{lem: cont depend 2} we know
\begin{equation}
\label{eq: thm discret apprx 1}
d_{\cT_{t_0,T},\infty}\left((\phi,\eta),(\phi^N,\eta^N)\right) \le \Big( \| \phi_0 - \phi_0^N \|_\infty  +   \|\eta_0-\eta_{t_0}^N\|^*_{\TV}  +	T \|\omega - \omega^N \|_\infty \Big) \txte^{C_1T}.
\end{equation}
By (A4), since $X$ is compact, we have $\phi_0(x)$ is uniformly continuous in $x$. Similar to the proof of Lemma \ref{lem: disc approx of the intial cond} \textbf{(i)}, we can show that
\begin{equation*}
\lim_{N\to \infty} \|\phi_0 - \phi_0^N \|_\infty =0.
\end{equation*}	
Moreover, due to the Lipschitz continuity assumption of $\omega$ and the definition of $\omega^N$ it is easy to check that
\begin{equation*}
\lim_{N\to \infty} \|\omega - \omega^N \|_\infty =0.
\end{equation*}
Combining these equations together with Lemma \ref{lem: disc approx of the intial cond} we obtain from \eqref{eq: thm discret apprx 1} the desired limit \eqref{eq: disc apr}.
\end{proof}


\section{An adaptive Kuramoto model}
\label{sec: main application}

Let $X=[0,1]$ equipped with the standard Borel $\sigma$-algebra, the Lebesgue measure $\mu_X$ and the partition $X^n_i:= [\frac{i-1}{n},\frac{i}{n})$, $i \in [n], n \in \mathbb{N}$. Consider the co-evolutionary Kuramoto-type models on $\mathbb{T}^N \times \mathbb{R}^{N \times N}$  from \cite{BVSY21},
\begin{subequations}\label{eq: ex: coevolution}
\begin{alignat}{2}
\dot{\phi}^N_i=&\omega+\frac{1}{N}\sum_{j=1}^NW^N_{ij}(t)\sin(\phi_i^N - \phi_j^N+a) , \quad \phi^N(0) = \phi^{N,0} \in \mathbb{T}^N,\label{eq: ex: coevolution-a}\\
\dot{W}_{ij}^N=&-\varepsilon(W_{ij}^N+\sin(\phi_i^N - \phi_j^N + b))\quad i,j \in [N]\label{eq: ex: coevolution-b}, \quad W^N(0) = W^{N} \in \mathbb{R}^{N \times N},
\end{alignat} 	
\end{subequations}
where  $a,b, \omega\in \mathbb{R}$ are real numbers, $\varepsilon >0$ is a small parameter and the initial conditions will be specified later (cf. (H2)). Note that this model is a special case of the system \eqref{coevolution}, for the choice $D,H:\mathbb{T}^2\to \mathbb{R}$, $D(u,v):= \sin(u- v +a)$, $H(u,v) := \sin(u- v +b)$ and intrinsic frequency $\omega$ which is a constant function.

System \eqref{eq: ex: coevolution} is motivated from neuroscience to model synaptic connections between periodically spiking neurons. Here $\phi_i$ models the phase of the $i$-th neuron and the weights $W_{ij}$ model the synaptic connections within the network of oscillators, which adapts according to the phase change as time evolves. The adaptation rule \eqref{eq: ex: coevolution-b} for the weights allows us to model lots of different synaptic plasticity scenarios. For instance, $\beta=0$ yields a Hebbian learning/plasticity rule meaning that synapses between neurons, which spike in sync, get strengthened. For $\beta= - \frac{\pi}{2}$ we have so-called spike-timing-dependent plasticity (STDP) meaning the synapses between $i$-th and $j$-th neuron gets strengthened, provided that the $j$-th neuron spikes before the $i$-th neuron (see \cite{Berner2019} and references therein). The continuum counterpart of \eqref{eq: ex: coevolution} is following integro-differential equation:
\begin{subequations}\label{characteristic-Eq-6}
\begin{alignat}{2}
\frac{\partial\phi(t,x)}{\partial t}=&\omega +\int_0^1\sin(\phi(t,x)- \phi(t,y) +a)\rd\eta_{t}^x(y),\\
\frac{\partial\eta_t^x(y)}{\partial t}=&-\varepsilon\eta_{t}^x(y)
-\varepsilon \sin(\phi(t,x)- \phi(t,y) + b) \mu_X(y),  \label{characteristic-Eq-6b}\\
\phi(t_0,x)=&\phi_0(x),\quad \eta_{t_0}^x=\eta_0^x,
\end{alignat}	
\end{subequations}
with
the initial conditions $(\phi_0,\eta_0)\in C(X,\mathbb{T}\times\cM(X))$. We assume:
\begin{itemize}
\item[(H1)]The initial measure $\eta_0 \in \cM_{\abs}(X)$ with
\begin{equation*}
W(x,y)=\frac{\rd \eta_0^x(y)}{\rd\mu_X(y)}, \quad x \in X,
\end{equation*}
is continuous in $x$ and $y$ satisfying
\begin{equation*}
\inf_{x,y\in X}W(x,y):= c_W >0
\end{equation*}	
\end{itemize}
The initial phases and the discrete/finite-dimensional edge weights are then given by
\begin{align}
\label{eq: ex: weigths-a}
\phi^{N,0}_i &:= N \int_{\frac{i-1}{N}}^{\frac{i}{N}}\phi_0(y) ~\rd y ,\\\label{eq: ex: weigths-b}
W^{N}_{i,j} &:= N^2 \int_{\frac{i-1}{N}}^{\frac{i}{N}} \int_{\frac{j-1}{N}}^{\frac{j}{N}}W(x,y) ~\txtd \mu_X(x) ~\txtd \mu_X(y), \quad  i,j \in [N].
\end{align}	
\begin{thm}\label{theo: main appl} Assume (H1). Let $0<T < \frac{1}{\epsilon}\ln(1 +c_W)$. Then there exists a unique solution $(\phi(t), \eta_t) \in C(\mathcal{T}_{t_0,T}\times X,\mathbb{T}\times\cM_{+}(X))$ (in the sense of Definition \ref{def: sol of the IVP}) of the continuum limit equation \eqref{characteristic-Eq-6}. Moreover, let $\phi^{N,0}$ and $W^N$ be defined in \eqref{eq: ex: weigths-a} and \eqref{eq: ex: weigths-b}, respectively; let $\phi^N$ and $\eta^N$ be defined in  \eqref{eq: weights in intro 1} and \eqref{eq: step graphon intro}, respectively. Then
\begin{equation*}
\lim_{N \to \infty}d_{\cT_{t_0,T},\infty}\left((\phi,\eta),(\phi^N,\eta^N) \right) = 0.
\end{equation*}	
\end{thm}
\begin{proof}
We first show the well-posedness. It is obvious that (A1)-(A3) and (A7) are fulfilled. Moreover, (A4) and (A6) both follow from (H1). (A5) follows from $0<T < \frac{1}{\epsilon}\ln(1 +c_W)$. Hence the conclusions follow immediately from Theorem \ref{thm A} and Theorem \ref{thm: discr approx}.
\end{proof}

\section{Discussion and outlook}
\label{sec: lim of the framework}

In this paper we use measure-valued (continuous) functions, so-called digraph measures introduced to network dynamics on graphs first in \cite{KuehnXu}, to represent the underlying graph with vertices in the compact space $X$. We have studied continuum limits for fully adaptive networks, which have seen recently gained a lot of interest in the natural sciences. We have established a result on well-posedness of a limiting integro-differential equation under very weak conditions and established a discrete-to-continuum approximation under additional hypotheses. The assumptions (A1)-(A4), which we used for well-posedness place very few restrictions on the underlying graph structures. Yet, assumption (A5) required here for the continuum limit approximation properties is stronger. Indeed, (A5) implies the underlying graph measure is absolutely continuous with a density $W$. Such $W$ is precisely a graphon \cite{L12}, a limit of a sequence of dense graphs. Moreover, convergence of absolutely continuous measures in total variation distance is equivalent to that of their densities in the $L_1$-norm induced metric. Hence (A5)-(A6) just means we can consider graphons with a uniform lower bound (for positivity) and such graphons are continuous. Indeed, except the technical positivity condition \eqref{condition-positivity} in (A5), graphons satisfying (A5)-(A6) form a dense set of all graphons on $X^2$, since the set of continuous functions are dense in $L^1(X^2;\mu_X\otimes\mu_X)$. This observation also means it is likely that one can further generalize our results to the entire space of positive $L^1$-graphons (satisfying the positivity condition), using additional approximation arguments. We do not believe this extension to require additional significant creativity, just the notation will become extremely involved and may obscure the main ideas, which we have laid out in this paper, to derive continuum limits for adaptive network dynamics. A similar remark also applies to typical slight modifications and extensions, which are possible with our methods. These extensions include (E1) closely related models on time-and state-dependent networks, e.g., the Cucker-Smale model with adaptive couplings or state dependent sensing \cite{HDP2020,Drfler2012}, (E2) existence in backward time of the continuum limit, (E3) lowering Lipschitz regularity to prove Peano-type existence theorems, and (E4) extending results to locally Lipschitz weight dynamics. From a viewpoint of most applications, which motivated the study of adaptive network dynamics in the first place, these extensions to our results are somewhat of secondary importance, e.g., second-order ODEs can be re-written as first-order systems, existence in backward time will follow upon reversing suitable integrals, Peano-type theorems follow from adaptations of classical ODE methods, and local Lipschitz conditions can be implemented using cut-offs.  

Yet, there is a key open problem, which is highly challenging mathematically and of major importance from an applied perspective: How far can we push discrete approximation results of the continuum limit to different lower-density graph limit objects, i.e., to intermediate density or sparse graphs? Well-posedness works already, and we have set up our framework via measure-theoretic arguments to allow for generalizations. For example, in this paper we used the total variation, which induces a strong topology in the space of finite signed measures. Next, it seems natural to ask if analogous or even better conclusions also hold in the weak topology in the space of finite positive measures induced by the bounded Lipschitz metric. Yet, the immense analytical difficulty level of trying to find a sharp boundary in the space of graph limits is apparent. For example, sparse graphs (such as the circular graphings discussed in \cite[Example 5.5]{GK20}; see also examples in \cite{KuehnXu}) or graphs of intermediate density (e.g., the spherical graphop discussed in \cite[Example 5.4]{GK20}; see also examples in \cite{KuehnXu}), though continuous in $x$ in bounded Lipschitz metric, may not be continuous in the total variation norm. Hence, it is likely that an extremely delicate choice of metric(s) is required to tackle various classes of sparse graphs. We are currently pursuing this crucial question regarding the interaction of density scales of graph limits with continuum, as well as mean-field, limits. However, we anticipate that this problem will require a very long-term effort including partial steps lowering density requirements gradually.

\medskip

\textbf{Acknowledgements:} MAG and CK gratefully thank the TUM International Graduate School of Science and Engineering (IGSSE) for support via the project ``Synchronization in Co-Evolutionary Network Dynamics (SEND)''. CK also acknowledges partial support by a Lichtenberg Professorship funded by the VolkswagenStiftung. CX acknowledges TUM Foundation Fellowship as well as the Alexander von Humboldt Fellowship.

\bibliographystyle{plain}
\bibliography{paper_submission}

\end{document}